\documentclass[12pt]{amsart}
\usepackage{amssymb}
\usepackage{amsmath}
\usepackage{amsthm}
\usepackage{times}
\usepackage{changepage}
\usepackage{caption}
\usepackage{geometry}
\usepackage[curve,all]{xy}
\xyoption{curve}
\xyoption{line}
\xyoption{arc}
\xyoption{color}
\usepackage{array}
\usepackage{enumitem}
\usepackage{graphicx}
\usepackage[dvipdfmx]{color}
\usepackage{color}
\usepackage{amsthm}
\newtheorem{theorem}{Theorem}[section]
\newtheorem{lemma}[theorem]{Lemma}
\newtheorem{corollary}[theorem]{Corollary}
\newtheorem{proposition}[theorem]{Proposition}

\usepackage{subfig}
\usepackage{hyperref}
 
\theoremstyle{definition}
\newtheorem{definition}[theorem]{Definition}
\newtheorem{example}[theorem]{Example}

\theoremstyle{remark}
\newtheorem{remark}[theorem]{Remark}

\newtheorem{convention}[theorem]{Convention}
\newtheorem{summary}[theorem]{Summary}

\numberwithin{equation}{section}

\newcommand{\bbC}{\mathbb{C}}

\newcommand{\bbP}{\mathbb{P}}

\newcommand{\bbZ}{\mathbb{Z}}

\def\Pic{{\mathrm{Pic}}}

\def\rank{{\text{rank}}}

\def\NS{{\rm{NS}}}

\def\deg{{\text{deg}}}

\def\I{{\text{I}}}

\def\IV{{\text{IV}}}

\begin{document}

\title{The 3-divisibility of divisors on K3 surfaces in characteristic 3}

\author{Toshiyuki Katsura}
\address{Graduate School of Mathematical Sciences, The University of Tokyo, Tokyo,
153-8914, Japan}
\email{tkatsura@g.ecc.u-tokyo.ac.jp}

\author{Matthias Sch\"utt}
\address{Institut f\"ur Algebraische Geometrie, 
Leibniz Universit\"at Hannover, Welfengarten 1, 30167 Hannover, Germany}

\address{Riemann Center for Geometry and Physics, Leibniz University
  Hannover, Appelstrasse 2, 30167 Hannover, Germany}
\email{schuett@math.uni-hannover.de}

\thanks{Research of the first author is partially supported by JSPS Grant-in-Aid 
for Scientific Research (C) No.26K06737. 
The second author's research is partly conducted in the framework of the research training
group GRK 2965: From Geometry to Numbers,
funded by DFG}


\date{April 28, 2026}

\begin{abstract}
We  describe the possible $3$-divisible $A_2^n$ configurations of smooth rational curves
on K3 surfaces in characteristic $3$ and fully classify the resulting triple covers.
\end{abstract}

\maketitle

\section{Introduction}

In the theory of algebraic surfaces, special configurations of smooth rational curves
can impose crucial geometric properties.
Notably, they can induce surprising divisibility properties in the Picard group
which in turn lead to the existence of particular covers of the given surface
(which by themselves can be quite subtle).
A key instance is provided by elliptic surfaces with torsion sections
as these translate into divisibilities among fibre components 
encoded in the theory of Mordell--Weil lattices \cite{SS}.
Special attention has been paid to the case of K3 surfaces
where the cover often again is K3, or an abelian surface.
In fact, the cases over $\bbC$ and in characteristic relatively prime to the order of divisibilty
are very well understood, see \cite{Barth1}, \cite{Barth2}, \cite{KS1}.
We should like to emphasize that they are closely related to finite symplectic automorphism groups,
following Xiao's take \cite{Xiao} on Nikulin's and Mukai's 
seminal classification results \cite{Nikulin-finite}, \cite{Mukai}
(see also \cite{DKeum}, \cite{OS} for the positive characteristic case).

The case where the order of divisibility and the characteristic agree is much more subtle
as our previous considerations in characteristic $2$ in \cite{KKS} reflect.
Here we turn to the 3-divisibility on K3 surfaces in characteristic $3$ 
which comes with yet different interesting features.
More precisely, we consider a K3 surface $X$ with an $A_2^n$ configuration 
of smooth rational curves which is 3-divisible (see Section \ref{s:def} for details).
We will show quite easily that $n=6$ or $9$ (see Corollary \ref{6,9}).
Then our main results are as follows.

\begin{theorem}
\label{thm6}
\begin{enumerate}
\item[(i)]
If $n=6$, then either $X$ has finite height $h\leq 4$ and Picard number 
satisfying $13\leq \rho(X)\leq 22-2h$, or 
it is supersingular of Artin invariant $\sigma\leq 7$.
\item[(ii)]
The triple cover is either rational or K3.
\item[(iii)]
Conversely, there are K3 surfaces of Picard number $13$ admitting a $3$-divisible configuration of 6 $A_2$,
and any supersingular K3 surface of Artin invariant $\sigma\leq 7$
admits a $3$-divisible configuration of 6 $A_2$.
\end{enumerate}
\end{theorem}


\begin{theorem}
\label{thm9}
\begin{enumerate}
\item[(i)]
If $n=9$, then  $X$ is supersingular of Artin invariant $\sigma\leq 5$.
\item[(ii)]
The triple cover is rational (whence $X$ is unirational).
\item[(iii)]
Conversely, any supersingular K3 surface of Artin invariant $\sigma\leq 5$
admits a $3$-divisible configuration of 9 $A_2$.
\end{enumerate}
\end{theorem}

We also discuss specific models of quasi-elliptic surfaces as quartics in $\bbP^3$
and derive some interesting consequences (cf.\ Theorems \ref{converse}, \ref{converse-6} and Corollary 
\ref{cor:converse}).

%
%

The paper is organized as follows.
The next section reviews basics on rational curves and establishes
fundamental  divisibility results.
Section \ref{s:triple} discusses more specifically triple covers,
leading to subtleties related to  normalization addressed in Section \ref{s:normal}.
The resolution of singularities is analysed in Section \ref{s:resolution},
the main result Theorem \ref{thm:B} setting a cornerstone for the proof of Theorems \ref{thm6}, \ref{thm9}.
After some explicit examples in Section \ref{s:ex},
the paper continues with the proofs of Theorem \ref{thm6} in Section \ref{s:pf6}
and, much shorter, of Theorem \ref{thm9} in Section \ref{s:pf9}.
It concludes with a general discussion of special projective models  in Section \ref{s:n=9}
with consequences  for both  the $n=6$ and the  $n=9$ case.

\begin{convention}
Throughout the paper, $k$ is an algebraically closed field of characteristic $3$
unless otherwise noted.
Rational curves are assumed to be irreducible,
and root lattices negative definite.
\end{convention}

\section{$(-2)$-curves on K3 surfaces}
\label{s:def}

Let $X$ be a K3 surface over $k$.
Let $C_i$ and $C_i'$ ($1 \leq i\leq n$) be nonsingular rational curves
with $C_i\cdot C'_i = 1$. We assume that the unions $C_i\cup C'_i$ do not intersect
each other. 
We assume that 
the divisor 
$$\sum_{i=1}^{n}(a_iC_i + a'_iC'_i) \;\;\;
(a_i, a'_i\in\bbZ; \; 0\leq a_i \leq 2, 0\leq a'_i \leq 2)
$$
is divisible by 3 in ${\rm Pic}(X)$.

Then there exists an invertible sheaf ${\mathcal M}$ such that
$$
    {\mathcal M}^{\otimes 3}\cong {\mathcal O}_X\left(\sum_{i=1}^{n}(a_iC_i + a'_iC'_i)\right).
$$
Since ${\Pic}(X)$ is torsion-free, ${\mathcal M}$ is unique up to isomorphism.
In this section we determine the possibilities for $a_i$, $a'_i$.
Since the case $a_i =a'_i = 0$ is superfluous, we exclude it in the sequel
so that $n$ is determined by the $a_i, a_i'$.

Since $C_i$ and $C'_i$ are nonsingular rational curves on a K3 surface, we have 
$C_i^2= -2$ and $C'^2_i= -2$.
We set ${\mathcal M}\cdot C_i = m_i$ and ${\mathcal M}\cdot C'_i = m'_i$. Then, we have
$$
\left\{
\begin{array}{cl}
    -2a_i + a'_i = 3m_i, \\
    a_i -2a'_i = 3m'_i.
\end{array}
\right.
$$
By $0\leq a_i \leq 2$, $0\leq a'_i \leq 2$ and $(a_i, a'_i)\neq (0, 0)$, 
we see either $(a_i, a'_i) = (1, 2)$ or $(2, 1)$.
Without loss of generality, we may assume $(a_i, a'_i) = (1, 2)$ for any $i = 1, 2, \ldots, n$;
this translates as
\begin{eqnarray}
\label{eq:3M}
   {\mathcal M}^{\otimes 3} \cong {\mathcal O}_X\left(\sum_{i=1}^{n}(C_i + 2C'_i)\right).
\end{eqnarray}
We compute the self-intersection number
\begin{equation}\label{NM}
{\mathcal M}\cdot {\mathcal M} = -\frac{2n}{3}.
\end{equation}

\begin{lemma}\label{3,6,9}
     $n =3$, $6$ or $9$.
\end{lemma}

\begin{proof}
The rank of N\'eron-Severi group ${\rm NS}(X)$ is smaller than or equal to 22.
The curves $C_i, C_i'$ generate a sublattice $A_2^n\subseteq\NS(X)$, whence $n\leq 11$.
If $n = 11$, ${\rm NS}(X)$ would thus be negative definite, 
but it is hyperbolic by the Hodge index theorem.
Therefore  $n \leq 10$. Since ${\mathcal M}^2$ is an (even) integer, 
(\ref{NM}) shows that
$n$ is divisible by 3.
The result follows from these facts.
\end{proof}

 \begin{lemma}\label{M}
$\dim {\rm H}^0(X, {\mathcal M}) =\dim {\rm H}^2(X, {\mathcal M}) =0$ and
$\dim {\rm H}^1(X, {\mathcal M}) = \frac{n}{3} -2$.
\end{lemma}
\begin{proof}
If ${\rm H}^0(X, {\mathcal M})\neq 0$, then there exists an effective divisor $D$ 
such that ${\mathcal M} \sim D$.
Then, we have $3D \sim {\mathcal M}^{\otimes 3} \sim \sum_{i=1}^{n}(C_i + 2C'_i)$. 
By $3D \cdot C'_i = -3 <0$ we see $C'_i \subset D$. 
Therefore, $D-\sum_{i=1}^nC'_i$ is an effective divisor. We denote it by $D'$.
By $(D + 2D')\cdot C_i = -2 <0$, we see $C_i \subset D' \subset D$.
Thus we obtain another effective divisor $D-\sum_{i=1}^n(C_i + C'_i)$.
In consequence, the non-zero divisor 
$D +(D-\sum_{i=1}^nC'_i) + (D-\sum_{i=1}^n(C_i + C'_i))$ is 
at the same time effective and, by spelling out $3D$, 
linearly equivalent to $0$,
which is impossible. Hence, 
we deduce ${\rm H}^0(X, {\mathcal M})= 0$.

By the Serre duality theorem, we get 
${\rm H}^2(X, {\mathcal M}) \cong {\rm H}^0(X, {\mathcal M}^{-1})$. 
Since ${\mathcal M}^{-3}\sim -3D<0$,
we have  $0 = {\rm H}^0(X, {\mathcal M}^{-1}) = {\rm H}^2(X, {\mathcal M})= 0$.
By the Riemann-Roch theorem, we deduce
$$
 -\dim {\rm H}^1(X, {\mathcal M}) = \chi({\mathcal M}) =\frac{{\mathcal M}\cdot {\mathcal M}}{2} + 2.
$$
Therefore, we conclude $\dim {\rm H}^1(X, {\mathcal M}) = \frac{n}{3} -2$.
\end{proof}

\begin{corollary}\label{6,9}
$n = 6$ or  $9$.
\end{corollary}
\begin{proof}
By Lemma \ref{M}, we have $\dim {\rm H}^1(X, {\mathcal M}) = \frac{n}{3} -2\geq 0$.
Therefore, we have $n \geq 6$ and the result follows from Lemma \ref{3,6,9}.
\end{proof}

\begin{remark}
The corollary can also be derived from the results about divisibilities in root lattices 
in \cite{S-nodal}.
It should thus come as no surprise that it is the same in any characteristic.
The geometric consequences, however, are rather different,
as recorded in Theorems \ref{thm6}, \ref{thm9}.
\end{remark}

\section{Triple covering}
\label{s:triple}

We denote by $P_i$ the intersection point of $C_i$ and $C'_i$ ($i= 1, 2, \hdots, n$).
We blow-up $X$ at $P_i$'s ($i= 1, 2, \hdots, n$) and denote the blowing-up by $\varphi$ and 
the obtained surface by $X'$:
$$
    \varphi : X' \longrightarrow X.
$$
Let $E_i = \varphi^{-1}(P_i)$ be the exceptional curve and $\bar{C}_i$ (resp. $\bar{C'_i}$)
be the proper transform of the curve $C_i$ (resp. $C'_i$), see Figure \ref{blowing-up}. 
They are nonsingular rational curves
and we have
$$
{\bar{C}_i}^2 = \bar{C'_i}^2 = -3, {E_i}^2 = -1.
$$
The dualizing sheaf $\omega_{X'}$ of $X'$ is given by
\begin{equation}\label{dualizing1}
    \omega_{X'} \cong \mathcal{O}_{X'}\left(\sum_{i=1}^{n}E_i\right).
\end{equation}
We have
$$
\varphi^{*}\left(\sum_{i=1}^{n}({C}_i + 2{C'_i})\right)= \sum_{i=1}^{n}(\bar{C}_i + 2\bar{C'_i} + 3E_i).
$$
Since $\sum_{i=1}^{n}(\bar{C}_i + 2\bar{C'_i} + 3E_i)$
is divisible by $3$ in ${\rm Pic}(X')$, $\sum_{i=1}^{n}(\bar{C}_i + 2\bar{C'_i})$ is also divisible by
$3$ in ${\rm Pic}(X')$, that is, there exists an invertible sheaf ${\mathcal L}$  on $X'$
such that
$$
    3{\mathcal L} \sim \sum_{i=1}^{n}(\bar{C}_i + 2\bar{C'_i}).
$$
We compute
\begin{equation}\label{L^2}
    {\mathcal L}\cdot {\mathcal L}= -\frac{5n}{3}.
\end{equation}

Let $\{U_i = {\rm Spec}(A_i)\}$ be an affine open covering of $X'$, 
and with respect to the covering $\{U_i\}$ let $\{f_{ij}\}$ be a cocycle  
which defines ${\mathcal L}$. Let $f_i =0$ be a local equation which defines
the effective divisor $\sum_{i=1}^{n}(\bar{C}_i + 2\bar{C'_i})$ on $U_i$. 
After adjusting the local equations $f_i$, we may assume that
\begin{equation}\label{1}
       f_{ij}^3 = f_i/f_j \quad \mbox{on}~U_i\cap U_j.
\end{equation}
Let $\tilde{U}_i \longrightarrow U_i$ be the covering defined by
\begin{equation}\label{2}
     z_i^3 = f_i.
\end{equation}
Then patching together by
$$
       z_i = f_{ij}z_j,
$$
we derive a purely inseparable covering of degree 3:
$$
    f :  Y \longrightarrow X'
$$
where $Y$ is projective, but not nonsingular. In fact, as we shall explore below, $Y$ is not even normal.

Since $f$ is locally given by 
$$
A_i \hookrightarrow A_i[z_i]/(z_i^3 - f_i) \cong A_i \oplus A_iz_i \oplus A_i{z_i}^2, 
$$
we see that
$$
  f_{*}({\mathcal O}_Y) \cong {\mathcal O}_X \oplus {\mathcal L}^{-1} \oplus {\mathcal L}^{-2}.
$$

Let $x_i$ (resp. $y_i$, resp. $w_i$) be a defining equation of $\bar{C}_i$ (resp. $\bar{C'_i}$,
resp. ${E}_i$). Since $\bar{C}_i$ (resp. $\bar{C'_i}$,
resp. ${E}_i$) is a nonsingular curve, we can take $x_i$ (resp. $y_i$, resp. $w_i$) as 
a part of a system of local parameters at each point on $\bar{C}_i$ (resp. $\bar{C'_i}$,
resp. ${E}_i$). Let $Q_i$ (resp. $R_i$) the intersection point of $\bar{C}_i$ and
${E}_i$ (resp. $\bar{C'_i}$ and ${E}_i$). 

\begin{figure}[htbp]
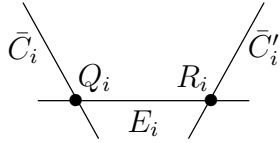

\xy
(-20,10)*{};
(40,18)*{};(50,0)*{}**\dir{-};
(62,0)*{};(72,18)*{}**\dir{-};
(42,5)*{};(70,5)*{}**\dir{-};
(56,2)*{E_i};
(40,12)*{\bar{C}_i};
(72,12)*{\bar{C}'_i};
(49.5,8)*{Q_i};
(62.5,8)*{R_i};
(47,5)*{\bullet};
(65,5)*{\bullet};
\endxy 
\caption{Blow-up at $P_i$}
\label{blowing-up}
\end{figure}
Then, $x_i$ and $w_i$ (resp. $y_i$ and $w_i$)
form a system of local parameters at $Q_i$ (resp. $R_i$). The defining equation $f_i$  of 
the effective divisor $\sum_{i=1}^{n}(\bar{C}_i + 2\bar{C'_i})$ at $Q_i$ (resp. $R_i$)
is given by 
$$f_i = u_i x_i \;\;\;  (\text{resp.}  \; f_i = u_i y_i^2)
$$
with a unit $u_i$ on $U_i$.
After shrinking $U_i$, if necessary, 
we may thus assume that the $U_i$ form an affine open covering such that $U_i$ 
contains $Q_i$ (resp. $R_i$), and 
$(x_i, w_i)$ (resp. $(y_i, w_i)$) are local parameters on $U_i$.

\begin{lemma}\label{L}
$\dim {\rm H}^0(X', {\mathcal L}^{\ell}) =\dim {\rm H}^2(X', {\mathcal L}^{\ell}) =0$,
$\dim {\rm H}^1(X', {\mathcal L}^{\ell}) = \frac{5\ell^2 n+ 3\ell n}{6} -2$ $(\ell = 1, 2)$.
\end{lemma}

\begin{proof}
If ${\rm H}^0(X', {\mathcal L})\neq 0$, then there exists an effective divisor $D$ 
such that ${\mathcal L}\sim D$.
Then, we have 
$$3D \sim {\mathcal L}^{\otimes 3} \sim \sum_{i=1}^{n}(\bar{C}_i + 2\bar{C'_i}).
$$
By $D \cdot \bar{C'_i} = -2 <0$, 
we see $\bar{C'_i} \subset D$. 
Therefore, $D'=D-\sum_{i=1}^n\bar{C'_i}$ is a (non-zero) effective divisor. 
Similarly, $D \cdot \bar{C_i} = -1 <0$ implies that $D''=D-\sum_{i=1}^n\bar{C_i}>0$.
%
In consequence, $2D' +D''>0$,
but by spelling this divisor  out, it is seen to be linearly equivalent to $0$, which is impossible.
Therefore, we have ${\rm H}^0(X', {\mathcal L})= 0$. 
In a similar way, we find ${\rm H}^0(X', {\mathcal L}^{\otimes 2})= 0$. 
Serre duality gives
$$
{\rm H}^2(X', {\mathcal L}^{\otimes \ell}) \cong {\rm H}^0(X', {\mathcal O}_{X'}(\sum_{i=1}^nE_i)\otimes{\mathcal L}^{-\ell}).
$$
Since ${\rm H}^0(X', {\mathcal O}_{X'}(3\sum_{i=1}^nE_i))\cong k$ (constants), we see
$$
{\rm H}^0(X', {\mathcal O}_{X'}(3\sum_{i=1}^nE_i)\otimes {\mathcal L}^{-3\ell})
\cong {\rm H}^0(X', {\mathcal O}_{X'}(3\sum_{i=1}^nE_i)\otimes 
{\mathcal O}_{X'}(-\ell\sum_{i=1}^{n}({\bar{C}_i + 2\bar{C'_i}})))= 0.
$$
Therefore, we obtain 
${\rm H}^0(X', {\mathcal O}_{X'}(\sum_{i=1}^nE_i)\otimes {\mathcal L}^{-\ell}) = 0$ 
 and
${\rm H}^2(X', {\mathcal L}^{-\ell})= 0$ $(\ell = 1, 2)$.
By the Riemann-Roch theorem, we compute
$$
  -\dim {\rm H}^1(X', {\mathcal L}^{\otimes \ell}) = \chi({\mathcal L}^{\otimes \ell}) 
  =\frac{{\mathcal L}^{\otimes \ell}\cdot ({\mathcal L}^{\otimes \ell} - \omega_{X'})}{2} + 2.
$$
Hence, we conclude $\dim {\rm H}^1(X', {\mathcal L}^{\otimes \ell}) = \frac{5\ell^2 n+ 3\ell n}{6} -2$.
\end{proof}
  
\begin{lemma}\label{L-1}
$\chi({\mathcal L}^{-\ell}) = \frac{- 5\ell^2 n+ 3\ell n}{6} +2$ and
$\dim {\rm H}^0(X', {\mathcal L}^{-\ell}) =\dim {\rm H}^2(X', {\mathcal L}^{-\ell}) =0$ ,
$\dim {\rm H}^1(X', {\mathcal L}^{-\ell}) = \frac{5\ell^2 n- 3\ell n}{6} -2$ $(\ell= 1, 2)$.
\end{lemma}

\begin{proof}
Since ${\rm H}^0(X', {\mathcal L}^{-3\ell})= 0$,
we have ${\rm H}^0(X', {\mathcal L}^{-\ell})= 0$ $(\ell = 1, 2)$. 
By the Serre duality
we have ${\rm H}^2(X', {\mathcal L}^{-\ell})
= {\rm H}^0(X', {\mathcal L}^{\otimes \ell}\otimes \mathcal{O}_{X'}(\sum_{i=1}^{n}E_i))$. 
In a similar way to the proof of Lemma \ref{L}, we see 
${\rm H}^0(X', {\mathcal L}^{\otimes \ell}\otimes \mathcal{O}_{X'}(\sum_{i=1}^{n}E_i)) = 0$ for $\ell=1,2$.
Using the Riemann-Roch theorem,
we have
$$
    - \dim {\rm H}^1(X', {\mathcal L}^{-\ell})= \chi({\mathcal L}^{-\ell}) = 
    \frac{{\mathcal L}^{-\ell}\cdot ({\mathcal L}^{-\ell} - \sum_{i=1}^{n}E_i)}{2} + 2 
    = \frac{- 5\ell^2 n+ 3\ell n}{6} +2.
$$
Hence, we complete our proof.
\end{proof}

\begin{proposition}
\label{prop:dualizing}
The dualizing sheaf $\omega_{Y}$ of $Y$ is given by
$$
        \omega_{Y} \cong f^*\left({\mathcal L}^{\otimes 2}\otimes{\mathcal O}_{X'}(\sum_{i=1}^{n}E_i)\right).
$$
\end{proposition}

\begin{proof}
Since $\omega_{X'} \cong {\mathcal O}_{X'}(\sum_{i=1}^{n}E_i)$,
by the duality theorem of a finite morphism (Hartshorne \cite[Exercise III. 6.10 (a)]{H},
for instance) we have
\begin{eqnarray}
f_{*}{\omega_{Y}}  &\cong & {\mathcal Hom}(f_{*}{\mathcal O}_{Y}, \omega_{X'}) 
\cong {\mathcal Hom}(f_{*}{\mathcal O}_{Y}, {\mathcal O}_{X'})\otimes \omega_{X'}
\nonumber\\
 &\cong & {\mathcal Hom}({\mathcal O}_{X'} \oplus {\mathcal L}^{-1} \oplus {\mathcal L}^{-2}, {\mathcal O}_{X'}) \otimes {\mathcal O}_{X'}(\sum_{i=1}^{n}E_i) 
 \nonumber\\
  &\cong & ({\mathcal O}_{X'} \oplus {\mathcal L} \oplus {\mathcal L}^{\otimes 2})
  \otimes {\mathcal O}_{X'}(\sum_{i=1}^{n}E_i) 
  \nonumber\\
  &\cong & ({\mathcal O}_{X'} \oplus {\mathcal L}^{-1} 
  \oplus {\mathcal L}^{-2})\otimes({\mathcal L}^{\otimes 2}
  \otimes {\mathcal O}_{X'}(\sum_{i=1}^{n}E_i)) 
    \label{eq:g_*omega}\\
  &\cong & f_{*}\{{\mathcal O}_{Y'} \otimes f^*({\mathcal L}^{\otimes 2}\otimes {\mathcal O}_{X'}(\sum_{i=1}^{n}E_i))\} \cong f_{*}\{f^*({\mathcal L}^{\otimes 2}\otimes {\mathcal O}_{X'}(\sum_{i=1}^{n}E_i))\}.
  \nonumber
\end{eqnarray}
Since $f$ is a finite morphism, we conclude 
$\omega_{Y} \cong f^*({\mathcal L}^{\otimes 2}\otimes{\mathcal O}_{X'}(\sum_{i=1}^{n}E_i))$.
\end{proof}

\begin{lemma}\label{q}
The irregularity $q(Y) = 0$.
\end{lemma}

\begin{proof}
Since $f$ is a purely inseparable morphism of degree 3, there is a dominant rational map
$X'^{(1/3)} \longrightarrow Y$. The universality of the Albanese variety gives
a surjective morphism ${\rm Alb}(X'^{(1/3)}) \longrightarrow {\rm Alb}(Y)$.
Since $X'^{(1/3)}$ is birationally equivalent to a K3 surface, 
we have $\dim {\rm Alb}(X'^{(1/3)}) =0$, and we conclude
$q(Y) = \dim {\rm Alb}(Y) = 0$.
\end{proof}

\begin{proposition}\label{chi}
  $\chi ({\mathcal O}_{Y}) =  - \frac{8n}{3} + 6$.
\end{proposition}
\begin{proof}
Since $f$ is a finite morphism, we have $R^if_{*}{\mathcal O}_{Y} = 0$ for $i \geq 1$.
Therefore we obtain
$$
\begin{array}{rl}
    \chi ({\mathcal O}_{Y}) &= \chi(f_{*}{\mathcal O}_{Y}) =
    \chi(({\mathcal O}_{X'} \oplus{\mathcal L}^{-1} \oplus{\mathcal L}^{-2})) \\
   & = \chi ({\mathcal O}_{X'}) + \chi({\mathcal L}^{-1}) + \chi({\mathcal L}^{-2}) 
   = 2 +\frac{- 5 n+ 3 n}{6} + 2 + \frac{- 20 n+ 6n}{6} + 2\\
   & = - \frac{8n}{3} + 6.
\end{array}
$$
\end{proof}

\begin{corollary}\label{cohomologyY}
$\dim {\rm H}^0(Y, {\mathcal O}_{Y})= \dim {\rm H}^2(Y, {\mathcal O}_{Y})=1$, 
$\dim {\rm H}^1(Y, {\mathcal O}_{Y})= \frac{8n}{3} -4$.
\end{corollary}
\begin{proof}
The assertion $\dim {\rm H}^0(Y, {\mathcal O}_{Y})= 1$ is clear. Since $g$ is a finite morphism,
we have $R^if_*\omega_Y = 0$ for $i \geq 1$. By \eqref{eq:g_*omega}, we have, using the finite morphism $f$ again,
$$
\begin{array}{rl}
   {\rm H}^2(Y, {\mathcal O}_{Y})\cong {\rm H}^0(Y, \omega_{Y})& \cong 
   {\rm H}^0(X', f_* \omega_{Y})\\
   &\cong {\rm H}^0(X', ({\mathcal O}_{X'} \oplus{\mathcal L} \oplus{\mathcal L}^{2})\otimes {\mathcal O}_{X'}(\sum_{i=1}^nE_i)) \cong k.
\end{array}   
$$
Finally, by Lemma \ref{chi} we conclude $\dim {\rm H}^1(Y, {\mathcal O}_{Y})=\frac{8n}{3} -4$.
\end{proof}

\section{Normalization}\label{Normalization}
\label{s:normal}

We consider a rational 1-form $\eta$ on $X'$ defined by
$$
    \eta = df_i/f_i  ~\mbox{on}~U_i.
$$ 
We denote the divisorial part by $(\eta)$
and the isolated singularity part by $\langle \eta \rangle$
(cf. Rudakov-Shafarevich \cite[Section 1]{RS0}). Then,
by an explicit calculation we have
\begin{equation}
\label{eq:eta}
 (\eta) = B - \sum_{i=1}^{n}(\bar{C}_i + \bar{C'_i}).
\end{equation}
with an effective divisor $B$. 

For the point $R_i$,
we consider a (partial) normalization:
$$
\left\{
\begin{array}{l}
    y_i= {Y}_i{Z}_i \\
    w_i={W}_i  \\
    z_i= {Z}_i  \\
\end{array}
\right..
$$
Then, we have 
\begin{equation}\label{Y'}
     {Z}_i =u_i{Y}_i^2. 
\end{equation}
For the coordinate rings, the (partial) normalization is locally given by
\begin{equation}\label{secondnormal}
    A_i[z_i][\frac{y_i}{z_i}] \supset A_i[z_i].
\end{equation}
We denote the obtained surface by $Y'$. Then we have a natural morphism
$$
\mu': Y' \longrightarrow Y.
$$
Using (\ref{secondnormal}), we have the conductor ideal
\begin{equation}\label{Ann}
{\rm Ann}(A_i[z_i][\frac{y_i}{z_i}]/A_i[z_i]) 
=(z_i) \subset A_i[z_i].
\end{equation}
From the equations (\ref{Y'}) obtained above, we see that 
over the divisor $\bar{C}_i + \bar{C'_i}+ E_i$
the surface $Y'$ is nonsingular. Since $B$ is a singular locus of codimension 1 in $Y'$, 
it does not intersect the divisor $\bar{C}_i + \bar{C'_i}+ E_i$.
There exist nonsingular rational curves $\tilde{C}_i$,
$\tilde{C'_i}$, $\tilde{E}_i$ on $Y''$ such that
\begin{equation}\label{pull-back}
(f\circ \mu')^{-1}(\bar{C}_i) = 3\tilde{C}_i, (f\circ \mu')^{-1}(\bar{C'_i}) 
= 3\tilde{C'_i}, 
(f\circ \mu')^{-1}(E_i) = \tilde{E}_i
\end{equation}
with self-intersection numbers
\begin{equation}\label{exceptional}
   \tilde{C}_i^2 = -1, \tilde{C'_i}^2 = -1, \tilde{E}_i^2 = -3.
\end{equation}

On $Y$, we define a Cartier divisor $C$ by $z_i = 0$ on $f^{-1}(U_i)$.
Then, by (\ref{Ann}) we have
$$
{\rm Ann}({\mu'}_{*}({\mathcal O}_{Y'})/{\mathcal O}_{Y})\cong {\mathcal O}_{Y}(-C)
    \subset {\mathcal O}_{Y}.
$$
Therefore, we have, as a ${\mu''}_{*}({\mathcal O}_{Y'})$-module,
$$
    {\mathcal Hom}({\mu'}_{*}({\mathcal O}_{Y'}), {\mathcal O}_{Y}) 
    \cong {\rm Ann}({\mu'}_{*}({\mathcal O}_{Y'})/{\mathcal O}_{Y})
    \otimes_{{\mathcal O}_{Y}}{\mu'}_{*}({\mathcal O}_{Y'})
    \cong {\mu'}_{*}({\mathcal O}_{Y'})\otimes_{{\mathcal O}_{Y}}{\mathcal O}_{Y}(-C)
$$
(cf. \cite[p.217]{CDL}). 
\begin{lemma}\label{Y''C}
${\mu'}_{*}\omega_{Y'}\cong 
f^*(({\mathcal L})^2\otimes {\mathcal O}_{X'}(\sum_{i=1}^n E_i))\otimes{\mathcal O}_{Y}(-C)
\otimes {\mu'}_{*}({\mathcal O}_{Y'})$.
\end{lemma}
\begin{proof}
Since $\mu'$ is a finite morphism, by a duality theorem for a finite morphism we have
$$
\begin{array}{rl}
   {\mu'}_{*}\omega_{Y'} & \cong {\mathcal Hom}({\mu'}_{*}({\mathcal O}_{Y'}), \omega_{Y}) \\
   &  \cong {\mathcal Hom}({\mu'}_{*}({\mathcal O}_{Y'}), {\mathcal O}_{Y}) \otimes  
   \omega_{Y}\\ 
   & \cong f^*(({\mathcal L})^2\otimes {\mathcal O}_{X'}(\sum_{i=1}^n E_i))
   \otimes {\mu'}_{*}({\mathcal O}_{Y'})\otimes{\mathcal O}_{Y}(-C).
\end{array}
$$
\end{proof}

Next we blow down $\tilde{C}_i$ and $\tilde{C'_i}$ $(i= 1, 2, \hdots, n)$. 
Then, the curve $\tilde{E}_i$
becomes an exceptional curve of the first kind. We also blow-down 
$\tilde{E}_i$ $(i= 1, 2, \hdots, n)$ and we denote the obtained surface by $Z'$:
$$
         h : Y' \longrightarrow Z'.
$$

We denote by $\omega_{Z'}$ the dualizing sheaf of $Z'$. Then we have
\begin{equation}\label{Y'Z}
  \omega_{Y'} \cong h^*(\omega_{Z'})\otimes {\mathcal O}_{Y'}(\sum_{i}^n(2\tilde{C}_i+ 2\tilde{C'_i}
  + \tilde{E}_i)).
\end{equation}
By (\ref{pull-back}) we have $f_{*}({\mathcal O}_{Y}(-3C)) = (f_{*}{\mathcal O}_{Y})\otimes {\mathcal O}_{X'}(-\sum_{i=1}^{n}2\bar{C'_i})$.
Therefore, we have
\begin{equation}\label{omegaY'}
\begin{array}{rl}
  ({f\circ \mu'})_{*}\omega_{Y'}^{\otimes 3} &\cong f_{*}(f^*(({\mathcal L})^6\otimes {\mathcal O}_{X'}(\sum_{i=1}^n 3E_i))\otimes{\mathcal O}_{Y}(-3C)
\otimes {\mu'}_{*}({\mathcal O}_{Y'}))\\
      & \cong (({\mathcal L})^6\otimes {\mathcal O}_{X'}(\sum_{i=1}^n 3E_i))
      \otimes f_*({\mathcal O}_{Y}(-3C)
\otimes {\mu'}_{*}({\mathcal O}_{Y'}))\\
      & \cong (({\mathcal L})^6\otimes {\mathcal O}_{X'}(\sum_{i=1}^n 3E_i))
      \otimes ({\mathcal O}_{X'}(-\sum_{i=1}^n 2\bar{C'_i})\otimes f_*({\mu'}_{*}({\mathcal O}_{Y'}))\\
      & \cong {\mathcal O}_{X'}(\sum_{i=1}^n (2\bar{C}_i + 4\bar{C'_i} +3E_i))\otimes {\mathcal O}_{X'}(-\sum_{i=1}^n 2\bar{C'_i})\otimes f_*({\mu'}_{*}({\mathcal O}_{Y'}))\\
       & \cong {\mathcal O}_{X'}(\sum_{i=1}^n (2\bar{C}_i + 2\bar{C'_i} +3E_i))\otimes
       f_*({\mu'}_{*}({\mathcal O}_{Y'}))\\
       & \cong ({f\circ \mu'})_{*}({f\circ \mu'})^*({\mathcal O}_{X'}(\sum_{i=1}^n (2\bar{C}_i + 2\bar{C'_i} +3E_i)))\\
       & \cong ({f\circ \mu'})_{*}({\mathcal O}_{Y'}(\sum_{i=1}^n (6\tilde{C}_i + 6\tilde{C'_i} +3\tilde{E}_i)))
\end{array}
\end{equation}
Thus we infer that
\begin{equation}\label{3}
     \omega_{Y'}^{\otimes 3}\cong {\mathcal O}_{Y'}(\sum_{i=1}^n (6\tilde{C}_i + 6\tilde{C'_i} +3\tilde{E}_i)).
\end{equation}

\begin{lemma}\label{3h}
    $h^*(\omega_{Z'}^{\otimes 3}) \cong {\mathcal O}_{Y'}$.
\end{lemma}
\begin{proof}
By (\ref{Y'Z}) and (\ref{3}), we have $h^*(\omega_{Z'})^{\otimes 3} \cong {\mathcal O}_{Y'}$.
\end{proof}



Regarding the relationship between $\deg \langle \eta \rangle$ and $B^2$, we have the following formula:

\begin{lemma}\label{<eta>}
$\deg \langle \eta \rangle = B^2 + 24 -3n$.
\end{lemma}
\begin{proof}
By Igusa's formula, we have
$$
   24 + n = c_2(X') = \deg \langle \eta\rangle + (\eta)\cdot \omega_{X'} - (\eta)^2.
$$
As we have already noted, $B$ does not meet $C_i$, $C_i'$ or $E_i$.
Since $(\eta)\cdot \sum_{i=1}^nE_i = -2n$ and $(\eta)^2 = B^2 - 6n$, the desired result follows.
\end{proof}

Finally, in this section, we provide some remarks concerning the special case $B = 0$.

\begin{lemma}\label{6}
If $B = 0$, then $n = 6$ and $\deg \langle \eta \rangle = 6$.
\end{lemma}
\begin{proof}
By Lemma \ref{<eta>}, we have $\deg \langle \eta \rangle = 24 -3n \geq 0$.
Therefore, we have $8 \geq n$. From Corollary \ref{6,9} we deduce $n= 6$ as stated.
Therefore, we have $\deg \langle \eta \rangle = 6$.
\end{proof}

\begin{lemma}\label{singularity}
Assume $B = 0$. Then the possible types of isolated singularities of $Y$  are as follows:
$$
6 A_2, E_6 + 3A_2, 2E_6, E_8 + 2A_2, \text{an elliptic singularity}.
$$
\end{lemma}
\begin{proof}
Let $P$ be an isolated singularity on $Y$. We may assume $Y$ is defined by
$z_i^3 = f_i$. Consider the completion $\hat{O}_{f(P)}$ 
of the local ring $O_{f(P)}$ at $f(P)$. Then we have $\hat{O}_{f(P)}\cong k[[x, y]]$
for suitable variables $x$, $y$. Since $P$ is a singular point, $f_i$
begins with terms of degree at least 2 in $k[[x, y]]$.
Moreover, any term which is a perfect cube can be absorbed by $z_i^3$. 
Therefore, we may assume that $f_i$ contains no terms of exponent divisible by 3.
By Lemma \ref{6}, we have $\dim k[[x, y]]/(f_{i,x}, f_{i,y})\leq 6$.

If the lowest degree of $f_i$ is greater than 3, then it is clear that 
$\dim k[[x, y]]/(f_{i,x}, f_{i,y})\geq 8$. 
If the lowest degree of $f_i$ is equal to 3, then, without loss of generality,
we may assume that the degree-3 term of $f_i$ is either $x^2y$ or $x^2y + xy^2$.
In both cases, the elements $y^i$ ($i = 0, 1, 2, \ldots$) remain linearly independent
in $\dim k[[x, y]]/(f_{i,x}, f_{i,y})$. Therefore, we conclude that
the degree of the lowest term of $f_i$ must be 2.

Therefore, after a suitable change of coordinates, 
we have two cases:
\begin{enumerate}
\item[$({\rm I})$] $f_i = xy + \mbox{higher terms}$,
\item[$({\rm II})$] $f_i = x^2 + \mbox{higher terms}$.
\end{enumerate}
In case (I), 
the singular point $P$
defined by $z_i^3 = xy + \mbox{higher terms}$ is a rational singularity of type $A_2$,
and $\dim k[[x, y]]/(f_{i,x}, f_{i,y})= 1$. In case (II),
since $k[[x, y]]/(f_{i,x}, f_{i,y})\leq 6$, after a suitable change of 
the system of local parameters of $k[[x, y]]$,
only the following three cases are possible:
\begin{enumerate}
\item[$({\rm i})$] $f_i = x^2 + y^4$,
\item[$({\rm ii})$] $f_i = x^2 + y^5$,
\item[$({\rm iii})$] $f_i = x^2 + y^7$.
\end{enumerate}
In case (i) (resp. (ii), resp. (iii)), the singularity is of type $E_6$ 
(resp. $E_8$, resp. an elliptic singularity),
and $\dim k[[x, y]]/(f_{i,x}, f_{i,y})= 3$ (resp. $4$, resp. $6$).
The lemma then follows from these results, since $\deg \langle \eta \rangle = 6$.
\end{proof}

\section{Resolution}
\label{s:resolution}

Let $\mu : \tilde{Y}\longrightarrow Y'$ be the normalization of $Y'$ and  $\nu :\tilde{\tilde{Y}} \longrightarrow \tilde{Y}$ be a resolution
of singularities of the surface $\tilde{Y}$.
We have a diagram 
$$
\begin{array}{ccccccccc}
 \tilde{\tilde{Y}} &\stackrel{\nu}{\longrightarrow}& \tilde{Y} &\stackrel{\mu}{\longrightarrow}&
  Y' &\stackrel{{\mu}'}{\longrightarrow}& Y&  & \\
      & &  & & \quad \downarrow h &   & \quad \downarrow f &  & \\
      & & & &  Z'& &X' &\stackrel{\varphi}{\longrightarrow} & X 
\end{array}
$$
By Artin \cite[(3.3)]{Artin}, we have the injective morphism between dualizing sheaves:
\begin{equation}\label{nu}
   \nu_{*}{\omega}_{\tilde{\tilde{Y}}} \hookrightarrow {\omega}_{\tilde{Y}}.
\end{equation}

We set ${\rm Ann}(\mu_*{\mathcal O}_{\tilde{Y}}/{\mathcal O}_{Y'})= {\mathcal I}$.
We see that if $B \neq 0$, then $Y'$ has non-isolated singularities. Therefore, we have
${\mathcal I}\neq 0$.
As before, we have
\begin{equation}\label{omegatildeY}
\begin{array}{rl}
   \mu_*\omega_{\tilde{Y}}& \cong {\mathcal Hom}(\mu_*{\mathcal O}_{\tilde{Y}}, \omega_{Y'})
   \cong {\mathcal Hom}(\mu_*{\mathcal O}_{\tilde{Y}}, {\mathcal O}_{Y'})\otimes \omega_{Y'}\\
   &\cong{\mathcal I}\otimes \omega_{Y'}\otimes \mu_{*}{\mathcal O}_{\tilde{Y}}
   \cong \mu_{*}\mu^{*}({\mathcal I}\otimes \omega_{Y'}).
\end{array}
\end{equation}
Therefore, we have
\begin{equation}\label{tildeY}
   \omega_{\tilde{Y}} \cong \mu^*({\mathcal I}\otimes \omega_{Y'}).
\end{equation}

\begin{lemma}\label{tildeYco}
If $B \neq 0$, then
${\rm H}^0(\tilde{Y}, \omega_{\tilde{Y}}) ={\rm H}^0(\tilde{Y}, \omega_{\tilde{Y}}^{\otimes 2})= 0$.
\end{lemma}
\begin{proof}
First we show ${\rm H}^0(\tilde{Y}, \omega_{\tilde{Y}}^{\otimes 6})= 0$. 
By (\ref{3}), we have 
$$
\mu^{*}\omega_{Y'}^{\otimes 6}\cong {\mathcal O}_{\tilde{Y}}(\sum_{i=1}^n (12\mu^{-1}(\tilde{C}_i) + 12\mu^{-1}(\tilde{C'_i}) +6\mu^{-1}(\tilde{E}_i))).
$$
Since $\mu^{-1}(\tilde{C}_i)$, $\mu^{-1}(\tilde{C'_i})$ and $\mu^{-1}(\tilde{E}_i)$ generate 
a negative definite lattice in ${\rm NS}(\tilde{Y})$,
we see
${\rm H}^0(\tilde{Y}, \mu^{*}\omega_{Y'}^{\otimes 6})$ is one-dimensional and
it consists of constants. Since $B\neq 0$, we see $\mu^{*}{\mathcal I} \neq (0)$ and 
the cycle ${\mathcal Z}$
defined by $\mu^{*}{\mathcal I}$ doesn't meet 
the divisor $\mu^{-1}(\tilde{C}_i) + \mu^{-1}(\tilde{C'_i}) +\mu^{-1}(\tilde{E}_i)$.
Therefore, any element $\theta$ in ${\rm H}^0(\tilde{Y}, \mu^*({\mathcal I}\otimes \omega_{Y''})^{\otimes 6})$ 
which is a constant must have zero along
${\mathcal Z}$. This means $\theta = 0$. Therefore, by (\ref{tildeY})
we have ${\rm H}^0(\tilde{Y}, \omega_{\tilde{Y}}^{\otimes 6})
={\rm H}^0(\tilde{Y}, \mu^*({\mathcal I}\otimes \omega_{Y'})^{\otimes 6})=0$. 
Hence, we have 
${\rm H}^0(\tilde{Y}, \omega_{\tilde{Y}}) =0$ and 
${\rm H}^0(\tilde{Y}, \omega_{\tilde{Y}}^{\otimes 2})= 0$.
\end{proof}



We are now ready to prove the key result in the way to  Theorems \ref{thm6}, \ref{thm9}:

\begin{theorem}
\label{thm:B}
\begin{enumerate}
\item[({\rm i})] If $B = 0$ and all singularities of $Y'$ are isolated 
rational singularities, then
$p_g({\tilde{\tilde{Y}}}) =P_2({\tilde{\tilde{Y}}}) = 1$, and we have $n=6$ and
$\deg \langle \eta \rangle = 6$. Moreover,
${\tilde{\tilde{Y}}}$ is birationally equivalent to a K3 surface.
\item[({\rm ii})] If $B \neq 0$ or $Y'$ has an elliptic singularity,
then $p_g({\tilde{\tilde{Y}}}) =P_2({\tilde{\tilde{Y}}}) = 0$. In this case, 
${\tilde{\tilde{Y}}}$ is a rational surface.
\end{enumerate}
\end{theorem}

\begin{proof}
Since $Y$ is defined by a single equation in a line bundle over $X'$, any isolated
singularity on $Y$ is Gorenstein and normal. 

Assume that $B = 0$ and 
all singularities of $Y'$ are isolated rational singularities. Then we have 
$\tilde{Y} = Y'$ and any isolated singularity
of $\tilde{Y}$ is locally isomorphic to the corresponding isolated singularity of $Y$. 
Therefore, they are Gorenstein and normal. Therefore, the dualizing sheaf $\omega_{\tilde{Y}}$
of $\tilde{Y}$ is an invertible sheaf. Moreover, by Lemma \ref{6}
we have $n= 6$, $\deg \langle \eta \rangle = 6$. 
Since we have $\tilde{Y} = Y'$, we identify $\tilde{Y}$ with $Y'$. 
By (\ref{3}) we have
$$
   \omega_{\tilde{Y}}^{\otimes 3} \cong {\mathcal O}_{\tilde{Y}}(\sum_{i=1}^n (6\tilde{C}_i + 6\tilde{C'_i} +3E_i)).
$$
If the singularities of $\tilde{Y}$ are all isolated and rational, then we have
$$
      \omega_{\tilde{\tilde{Y}}} \cong \nu^*\omega_{\tilde{Y}}.
$$
Therefore, we have 
\begin{equation}\label{3tildetildeY}
\omega_{\tilde{\tilde{Y}}}^{\otimes 3} \cong 
{\mathcal O}_{\tilde{\tilde{Y}}}(\sum_{i=1}^n (6\nu^{-1}(\tilde{C}_i) + 6\nu^{-1}(\tilde{C'_i}) +3\nu^{-1}(E_i))).
\end{equation}
We can blow-down the divisors $\nu^{-1}(\tilde{C}_i)$, $\nu^{-1}(\tilde{C'_i})$ and 
$\nu^{-1}(E_i)$
$(i= 1, 2, \ldots, n)$ successively, and we denote by $Z$ the obtained surface as in the case of $Z'$.
We denote the blow-down by $\psi$:
$$
        \psi : \tilde{\tilde{Y}}  \longrightarrow Z.
$$
Then we have 
$$
\omega_{\tilde{\tilde{Y}}} \cong 
\psi^*(\omega_{Z})\otimes {\mathcal O}_{\tilde{\tilde{Y}}}(\sum_{i=1}^n(2\nu^{-1}(\tilde{C}_i) + 2\nu^{-1}(\tilde{C'_i}) +\nu^{-1}(E_i))).
$$
Using (\ref{3tildetildeY}), we have
$$
     \psi^{*}\omega_{Z}^{\otimes 3} \cong {\mathcal O}_{\tilde{\tilde{Y}}}.
$$
Since $\psi^{*}$ on ${\rm Pic}(Z)$ is injective, we see $\omega_{}^{\otimes 3} \cong {\mathcal O}_{Z}$. Therefore, the Kodaira dimension of $Z$
 is $\kappa (Z) = 0$. 
Since $Z$ is birationally equivalent to $Y$, by Lemma \ref{q} we have $q(Z)= 0$. 
By the Enriques--Kodaira classification,
$Z$ is either a K3 surface or an Enriques surface. 
Since Enriques surfaces satisfy $\omega_Z^{\otimes 3}\not\cong\mathcal O_Z$ 
outside characteristic $2$,
we conclude that $Z$ is a K3 surface and $\omega_{Z}\cong {\mathcal O}_{Z}$.
\smallskip

Now, we assume $B \neq 0$. By (\ref{nu}) we have an injection
$$
   {\rm H}^0(\tilde{\tilde{Y}}, \omega_{\tilde{\tilde{Y}}}^{\otimes 2}) \hookrightarrow 
   {\rm H}^0({\tilde{Y}}, \omega_{\tilde{Y}}^{\otimes 2}).
$$
By Lemma \ref{tildeYco}, we have ${\rm H}^0(\tilde{\tilde{Y}}, \omega_{\tilde{\tilde{Y}}}^{\otimes 2}) = 0$. 
In the same way as above, we have $q(\tilde{\tilde{Y}}) = 0$.
Therefore, by Castelnuovo's criterion of rationality, we see that $\tilde{\tilde{Y}}$ is rational.

\smallskip

Finally, assume that $B=0$ and that $Y'$ has an elliptic singularity.
Note that in this case we have $\tilde{Y} = Y'$.
For the resolution $\nu : \tilde{\tilde{Y}} \longrightarrow \tilde{Y}$,
we denote by $\bigcup_{i = 1}^{k}G_i$ the support of the exceptional divisors
of $\nu$, where each $G_i$ is an irreducible component.
By Hidaka--Watanabe \cite[Lemma 1.1]{HW}, we have
$$
 \omega_{\tilde{\tilde{Y}}} \sim \mu^{*}\omega_{\tilde{Y}} - \sum_{i= 1}^{k}r_iG_i \quad (r_i \geq 0).
$$
Therefore,  
$$
\begin{array}{rl}
\omega_{\tilde{\tilde{Y}}}^{\otimes 6}& \sim \mu^{*}\omega_{\tilde{Y}}^{\otimes 6} - \sum_{i= 1}^{k}6r_iG_i \quad (r_i \geq 0)\\
    &  \sim {\mathcal O}_{\tilde{\tilde{Y}}}(\{\sum_{i=1}^n (\mu\circ \nu)^{*}(12\tilde{C}_i + 12\tilde{C'_i} +6\tilde{E}_i)- \sum_{i= 1}^{k}6r_iG_i \}).
\end{array}
$$
The exceptional divisors $G_i$ do not intersect $(\mu\circ \nu)^{*}\tilde{C}_j$,
$(\mu\circ \nu)^{*}(\tilde{C'_j})$ or  $(\mu\circ \nu)^{*}\tilde{E}_j$ for any $j = 1, 2, \ldots, k$.
Since $\tilde{Y}$ has a non-rational isolated singular point, there exists some $i$ with $r_i > 0$
by Hidaka--Watanabe \cite{HW}. Moreover, as in Section \ref{Normalization},
we can contract the curves $(\mu\circ \nu)^{*}\tilde{C}_j$,
$(\mu\circ \nu)^{*}\tilde{C'_j}$ and $(\mu\circ \nu)^{*}\tilde{E}_j$ onto a nonsingular surface.
It follows that ${\rm H}^0(\tilde{\tilde{Y}}, \omega_{\tilde{\tilde{Y}}}^{\otimes 6}) = 0$.
Hence, 
$$
   p_g(\tilde{\tilde{Y}}) = \dim {\rm H}^0(\tilde{\tilde{Y}}, \omega_{\tilde{\tilde{Y}}}) = 0, P_2(\tilde{\tilde{Y}})={\rm H}^0(\tilde{\tilde{Y}}, \omega_{\tilde{\tilde{Y}}}^{\otimes 2}) = 0.
$$
Since we also have $q(\tilde{\tilde{Y}})=0$ as shown above, we conclude by Castelnuovo's criterion of rationality
that $\tilde{\tilde{Y}}$ is rational.
\end{proof}

\begin{corollary}\label{n=9}
If $n = 9$, then $X$ is a supersingular K3 surface.
\end{corollary}
\begin{proof} 
Lemma \ref{6} implies that $B\neq 0$. 
Hence $X$ is unirational by Theorem \ref{thm:B}. Therefore, $X$ is supersingular.
\end{proof}

\begin{summary}
\label{summary}
With Theorem \ref{thm:B} in place,
we have derived the structural results \textit{(ii)}
of Theorems \ref{thm6}, \ref{thm9}.
Note that this is quite contrary to the complex case
where the cover always is either K3 or abelian.
\end{summary}

\section{First examples}
\label{s:ex}

In this section, we give some examples of K3 surfaces $X$  which have 
interesting divisors divisible by $3$
in ${\rm Pic}(X)$.
In fact, our examples will be independent of the characteristic and are fairly standard.

\begin{example}[$n=6$ on double sextics]
\label{ex:Shimada}
Let $f,g\in k[x,y,z]$ be homogeneous polynomials of degree $3$ resp.\ $2$
and consider the double cover of $\bbP^2$ defined by
\[
w^2 + wf = g^3.
\]
Generally, this has 6 isolated $A_2$ singularities whose resolution is a K3 surface $X$.
As recorded in \cite{Shimada}, their configuration is $3$-divisible,
effectively given by the two pre-images of the curve given by $g=0$ in $\bbP^2$. 
\end{example}

\begin{example}[$n=9$ from genus one fibration with three $\IV^*$]
\label{ex:3E_6}
We consider a genus one fibration on a  K3 surface $\varphi : X \longrightarrow {\bf P}^1$
defined by
$$
        y^2 + t^2 (t-1)^2 y = x^3.
$$ 
Then this has 3 singular fibres of Kodaira type $\IV^*$ at $t=0,1,\infty$.
%
%
Express them as
$F_i =3 C_0^{(i)} + 2\sum_{j= 1}^{3}C_j'^{(i)} + \sum_{j= 1}^3C_j^{(i)}$ $(i = 1, 2, 3)$
and  take a general fiber $F$.
We set $D = \sum_{i = 1}^{3}\sum_{j= 1}^{3}(2C_j'^{(i)} + C_j^{(i)})$. Then we have
$$
    D \sim 3(F - \sum_{i=1}^{3}C_0^{(i)}).
$$
Therefore, $D$ is divisible by 3 in ${\rm Pic}(X)$. This gives an example with $n=9$ in any characteristic
(which is quasi-elliptic, with another reducible fibres of type $\IV$ at $t=-1$, in characteristic $3$, and otherwise elliptic).
\end{example}

\begin{example}[$n=6$ from elliptic fibration with two $\IV^*$]
\label{ex:2E_6}
Let $X$ be a K3 surface admitting a genus one fibration admitting (at least) two fibres of type $\IV^*$.
Writing them as above, we infer that
\[
D = \sum_{j= 1}^{3}(2C_j'^{(1)} + C_j^{(1)}) - \sum_{j= 1}^{3}(2C_j'^{(2)} + C_j^{(2)}) \sim 3(C_0^{(2)}-C_0^{(1)}).
\]
Adding $3\sum_{j= 1}^{3}(C_j'^{(2)} + C_j^{(2)})$, we obtain that
\[
D' = \sum_{j= 1}^{3}(2C_j'^{(1)} + C_j^{(1)}) + \sum_{j= 1}^{3}(C_j'^{(2)} + 2C_j^{(2)})
\]
is 3-divisible in ${\rm Pic}(X)$.
\end{example}

%

\begin{example}[$n=6$ from elliptic fibration with $3$-torsion section]
\label{ex:P}
Consider a jacobian elliptic K3 surface $X$ with a 3-torsion section $P$.
In general, $X$ admits 6 singular fibres of type $\I_3$ 
which are met by $P$ in a different component than by the zero section $O$.
Then the theory of Mordell--Weil lattices \cite{SS} gives a 3-divisible $6A_2$ configuration as follows:
say that $O$ meets $\Theta_0^{(i)}$, $P$ meets $\Theta_1^{(i)}$, and $-P$ meets $\Theta_2^{(i)}$.
Denoting a fibre by $F$, we know the linear equivalence
\[
P-O-2F \sim - \frac 13 \sum_{i=1}^6 (2\Theta_1^{(i)}+\Theta_2^{(i)}),
\]
showing exactly what we want.
\end{example}

\begin{remark}
On quasi-elliptic fibrations, also $n=9$ can be realized by way of suitable $3$-torsion sections,
see Corollary \ref{m}.
\end{remark}

\begin{remark}
Example \ref{ex:2E_6} can be obtained from Example \ref{ex:P}
by (separably)
quotienting by the automorphism given fibrewise by translation by $P$.
We leave it to the reader to work out the inseparable triple covers given by the $3$-divisible $6A_2$ configurations.
\end{remark}

%

\section{Proof of Theorem \ref{thm6}}
\label{s:pf6}

\subsection{Proof of Theorem \ref{thm6} (ii)} 
This follows from Theorem \ref{thm:B},
as recorded in Summary \ref{summary}.
\qed

\subsection{Proof of Theorem \ref{thm6} (i)} 
\label{ss:pf-i}

Combining the $6A_2$ configuration with an ample class, we infer that $\rho(X)\geq13$.
Since finite height $h$ implies $\rho\leq 22-2h$, we deduce that $h\leq 4$ as claimed.

The  $6A_2$ configuration generates a rank 12 sublattice  $L\subset\NS(X)$ of determinant $81$,
due to the assumed $3$-divisibility, i.e.\ $L$ contains $6A_2$ as an index 3 sublattice.
If $X$ is supersingular, then the finite index sublattice decomposition
\begin{eqnarray}
\label{eq:L}
L \oplus L^\perp\subseteq\NS(X)
\end{eqnarray}
has the following implication for the 3-lengths $l$ of the discriminant groups $A_L=L^\vee/L$ etc:
\begin{eqnarray}
\label{eq:l}
2\sigma(X) = l(A_{\NS(X)}) \leq l(A_L) + l(A_{L^\perp}) \leq 4 + \rank(L^\perp) = 14.
\end{eqnarray}
Here, we use the fact that $\NS(X)$ is 3-elementary, 
more precisely 
$$A_{\NS(X)}=\NS(X)^\vee/\NS(X)\cong(\bbZ/3\bbZ)^{2\sigma}.
$$
%
Hence $\sigma\leq 7$ as stated.
\qed

\subsection{Proof of Theorem \ref{thm6} (iii)} 
\label{ss:pf-iii}

The existence statement for finite height, with  $\rho=13$,
can be deduced from Example \ref{ex:Shimada}, as recorded in the following characteristic-free lemma:

\begin{lemma}
\label{lem:7-d}
Consider the K3 surfaces arising from Example \ref{ex:Shimada}.
These form a 7-dimensional family whose very general member $X$ has $\rho(X)=13$.
\end{lemma}

\begin{proof}
The dimension of the family was recorded in \cite[Remark 6.3]{Shimada} for the case of the complex numbers,
building on \cite{Degtyarev},
but the statement remains valid in any characteristic.
It follows from the moduli theory of lattice polarized K3 surfaces
that the very general member  $X$ either has $\rho(X)=13$ or, in characteristic $p$,
$X$ is supersingular with $\sigma(X)\geq 8$.
In either case, in the notation of \ref{ss:pf-i}, we have a primitive embedding
\[
M = A_1^- \oplus L \hookrightarrow \NS(X)
\]
(which is an isometry if $\rho(X)=13$).
To rule out the supersingular case, we consider the orthogonal complement $M^\perp$
analogous to \eqref{eq:L} to find, as in \eqref{eq:l},
\[
2\sigma \leq l(A_M) + \rank(M^\perp) =
\begin{cases} 1 + 9 = 10 & p=2\\
4 + 9 =13 & p=3\\
9 & p>3.
\end{cases} 
\]
Hence the lemma follows in any characteristic (and in fact $\sigma\leq 6$ in characteristic $3$).
\end{proof}

For the supersingular case, 
let $\Lambda_\sigma$ denote the N\'eron--Severi lattice of a supersingular K3 surface
of Artin invariant $\sigma$ in characteristic $3$.
This is unique up to isometry (cf. Nikulin \cite{Nikulin} or 
Rudakov-Shafarevich \cite[Theorem, p.\ 1477]{RS1}).
For $\sigma=7$, it  can be expressed as
\begin{eqnarray}
\label{eq:sigma=7}
\Lambda_7 \cong U(3) \oplus L \oplus E_8(3)
\end{eqnarray}
where $U$ denotes the hyperbolic plane and the entries in brackets indicate that the intersection form is scaled by the entry
while $L\supset 6A_2$ is the index 3 overlattice from \ref{ss:pf-i}
which is generated by the usual 3-divisible element in $6A_2$
(as in Example \ref{ex:P}).

A primitive isotropic vector $E$ from the summand $U(3)$ defines a genus one fibration on $X$
which is elliptic 
because the root sublattice $R=6A_2$ of $E^\perp/\bbZ E = L \oplus E_8(3)$ has rank $12$ only.
In fact, as the root sublattice  encodes the reducible fibres by \cite[Lemma 2.2]{Kondo-elliptic},
we infer from $R=6A_2$ that there are 6 fibres of types $\I_3$ or $\IV$,
and the divisibility yielding $L$ provides us exactly with the desired $6A_2$ configuration given by suitable fibre components.
More precisely, there may be reflections involved in converting $E$ to a fibre class 
(possibly after changing sign) and the standard generators of the $A_2$ summands to 
fibre components (thus smooth rational curves), 
but any reflection preserves the divisibility property, of course.

For smaller $\sigma\leq 6$, one can start with $\Lambda_6\cong U \oplus L \oplus E_8(3)$, arguing exactly as above,
or consider the following hierarchy of lattices 
which give genus one fibrations featuring two fibres of type $\IV^*$ each:
$$
\begin{array}{l|c}
\sigma & \Lambda_\sigma\\
\hline
6 & U(3) \oplus 2E_6 \oplus E_8(3)\\
5 & U \oplus 2E_6 \oplus E_8(3)\\
4 &
U(3) \oplus 2E_6 \oplus 4 A_2\\
3
& 
U \oplus 2E_6 \oplus 4 A_2\\
2 &
U(3) \oplus L' \oplus A_2\\
1 & 
U \oplus L' \oplus A_2
\end{array}
$$
Here $L'\supset3E_6$ is the standard index 3 overlattice with the same set of roots
(as generated by the $3$-torsion section in Example \ref{ex:3E_6}),
and the first two fibrations are elliptic while the last 4 are quasi-elliptic
(the last one is given by Example \ref{ex:3E_6}).
Since every fibration admits at least two fibres of type $\IV^*$,
it supports a 3-divisible $6A_2$ configuration by Example \ref{ex:2E_6}.
This completes the proof of Theorem \ref{thm6} (iii).
\qed

\section{Proof of Theorem \ref{thm9}}
\label{s:pf9}

\subsection{Proof of Theorem \ref{thm9} (ii)} 
Again, this follows from Theorem \ref{thm:B}, cf.\ Summary \ref{summary}.
\qed

\subsection{Proof of Theorem \ref{thm9} (i)} 
\label{ss:pf9-i}

By Corollary \ref{n=9}, $X$ is supersingular.
Let $L$ denote the primitive closure of $9A_2\subset\NS(X)$,
a rank 18 lattice of determinant of absolute value at most $3^7$.
Arguing with the analogue of \eqref{eq:L},
we find  along the same lines that
\[
2\sigma = l(A_{NS(X)}) \leq l(A_L) + l(A_{L^\perp}) \leq 7 + \rank(L^\perp) = 11.
\]
Hence $\sigma\leq 5$ as stated.
\qed

\subsection{Proof of Theorem \ref{thm9} (iii)} 

For the converse direction, we argue with the index $3$ overlattice $L\supset 9A_2$
obtained by dividing a standard vector with $n=9$ by 3 as in \eqref{eq:3M}
Then the following representations give the claim:
$$
\begin{array}{l|c}
\sigma & \Lambda_\sigma\\
\hline
5 & U(3) \oplus L \oplus A_2\\
4 &
U \oplus L \oplus A_2\\
3
& 
U(3) \oplus 3E_6 \oplus A_2\\
2 &
U \oplus 3E_6 \oplus A_2\\
1 & 
U \oplus L' \oplus A_2
\end{array}
$$
Here we use Example \ref{ex:3E_6} for $\sigma\leq 3$,
and $L'\supset 3E_6$ again denotes the index 3 overlattice with same root sublattice 
(as in \ref{ss:pf-iii}).
\qed

\subsection{Comments of indivisible $A_2^9$ configurations}


We should like to highlight that the decomposition for Artin invariant $\sigma=6$,
\[
\Lambda_6 \cong U(3) + 10A_2,
\]
shows along the same lines as above that there can be $A_2^9$ configurations
on supersingular K3 surfaces in characteristic $3$ which involve no divisibility at all.
This is in stark contrast to any other characteristic where $3$-divisibility is automatic.
To see this, one can argue similarly to the proof of Lemma \ref{lem:7-d}.

In fact, indivisible $A_2^9$ configurations in characteristic $3$ may only occur on 
supersingular K3 surfaces (of Artin invariant $\sigma\leq 6$).
Otherwise, we could lift the K3 surface of finite height 
together with its entire $\NS$ to characteristic zero
whence the divisibility would follow from the above sketch.


\section{Quasi-elliptic surfaces with 10 singular fibers of type ${\rm IV}$}
\label{s:n=9}

\subsection{Structure}
\label{structure}

In this subsection, we consider the case where the quasi-elliptic surface 
$$\varphi : X \longrightarrow {\bf P}^1$$
has 10 singular fibers
of type ${\rm IV}$.
For $i= 1, 2, \ldots, 10$ and $j = 1, 2, 3$, let $E_j^{(i)}$ 
denote the irreducible components of the $i$-th singular fiber,
so that $E_1^{(i)}$, $E_2^{(i)}$, $E_3^{(i)}$ together form this fiber. 
Since $\rank {\rm NS}(X) = 20$, there are no other reducible fibers.
We denote by $F$ a general fiber and by $\xi$ the cusp locus of the quasi-elliptic fibration.
$\xi$ is a nonsingular rational curve, 
and the restriction morphism $\varphi\vert_{\xi} : \xi \longrightarrow {\bf P}^1$
is purely inseparable of degree 3 (Bombieri-Mumford \cite{BM}).
Moreover, $\xi$ passes through the common intersection point of $E_1^{(i)}$, $E_2^{(i)}$ and $E_3^{(i)}$
for each $i$ (Ito \cite{I1}).

\begin{theorem} 
\label{thm:H}
The divisor $H=\xi + F$ is very ample, and the surface $X$ is embedded in
${\mathbf P}^3$  as a quartic surface.
\end{theorem}

\begin{proof}
Let $P$, $Q$ be two points on different fibers. Then, there exists a section 
$s \in {\rm H}^0(X, {\mathcal O}_X(F))\subset {\rm H}^0(X, {\mathcal O}_X(\xi + F))$
such that $s(P) = 0$ and $s(Q) \neq 0$. Suppose that $P$ and $Q$ lie on the same fiber $F_0$.
Since $\xi + F$ is linearly equivalent to $\xi + F_0$, we may assume that $F = F_0$.
Consider the exact sequence
$$
 0 \rightarrow {\mathcal O}_X(-F) \longrightarrow {\mathcal O}_X \longrightarrow {\mathcal O}_F \rightarrow 0.
$$
Tensoring with ${\mathcal O}_X(\xi + F)$, we obtain an exact sequence
$$
 0 \rightarrow {\mathcal O}_X(\xi) \longrightarrow {\mathcal O}_X(\xi + F) \longrightarrow {\mathcal O}_F(\xi + F) \rightarrow 0.
$$
Since $\xi^2 = -2$, we have ${\rm H}^0(X, {\mathcal O}_X(\xi)) \cong k$. By the Riemann-Roch theorem, we have ${\rm H}^1(X, {\mathcal O}_X(\xi)) \cong 0$.
Therefore, the natural homomorphism ${\rm H}^0(X, {\mathcal O}_X(\xi + F)) \longrightarrow {\rm H}^0(F, {\mathcal O}_F(\xi + F))$ is surjective.
By the adjunction formula, the dualizing sheaf of $F$ is trivial. By $\deg {\mathcal O}_F(\xi + F)) = F\cdot(\xi + F) = 3$, we see that
the invertible sheaf ${\mathcal O}_F(\xi + F)$ is very ample on $F$. Hence, there exists a section $s \in {\rm H}^0(X, {\mathcal O}_X(\xi + F))$
such that $s(P) = 0$ and $s(Q) \neq 0$. It follows that the linear system $\vert \xi + F\vert$ is base point free, and 
the associated morphism $\psi_{\vert \xi + F\vert}$ is injective.

Therefore we can apply \cite[Thm.\ 5.2]{SD}
to deduce that $H$ is very ample.
To this end, it suffices to check that $H.C\geq 3$ for any irreducible curve $C\subset X$ with $p_a(C)=1$.
Indeed, if  $C$ is a fibre of $\varphi$, then $H.C=\xi.C=3$ as required.
Otherwise, $C$ is a multisection of $\varphi$, different from $\xi$ and not  a section, 
since either would be smooth rational. If $C$ has index $C.F\geq 3$, then
$H.C\geq F.C\geq 3$ as desired. Thus only the case of a bisection remains
(whence the fibration $\varphi$ admits a section).
More precisely, for $H$ not to be very ample, the bisection $C$ would have to be disjoint from $\xi$
(so that $H.C=F.C=2$). But then the genus one fibration induced by $|C|$ would have a fibre
containing not only $\xi$, but also all components of the reducible fibres of $\varphi$
which are not met by $C$. As this amounts to at least one component per $\IV$ fibre,
there would be at least ten branches attached to $\xi$ which is absurd in view of Kodaira's and Tate's classification
of singular fibres.
Hence $C$ cannot be a bisection and $H.C\geq 3$ follows as required.
\end{proof}

Let $H$ be a general hyperplane in ${\bf P}^3$. Then we have $\xi \cdot H = \xi\cdot (\xi + F) = 1$. Hence the image of $\xi$
in ${\bf P}^3$ is a line. Similarly, $E_j^{(i)} \cdot H = E_j^{(i)} \cdot (\xi + F) = 1$ and therefore the image of $E_j^{(i)}$ in ${\bf P}^3$ is also a line.
Fix a component  $E_j^{(i)}$. Then $E_j^{(i)}$ intersects $\xi$ transversely. Let $H$ be the hyperplane containing both $\xi$  and $E_j^{(i)}$.
Then $H \cap X$ is linearly equivalent to $\xi + F$. Since $\xi + E_j^{(i)} \subset H \cap X$, we obtain $H \cap X = \xi + \sum_{j=1}^3E_j^{(i)}$.
Let $H_0$ and $H_1$ be two distinct hyperplanes containing $\xi$. We consider the pencil generated by $H_0$ and $H_1$, 
namely  $tH_0 + sH_1$.
Summarizing the above discussion, we have the following proposition.
\begin{proposition}
The pencil $tH_0 + sH_1$ induces the quasi-elliptic fiberation $\varphi : X \longrightarrow {\bf P}^1$.
\end{proposition}

\begin{remark}
In the case of the Fermat quartic surface, the structure of quasi-elliptic fibration is described in \cite{K}
and in \cite{RS-112}.
In this case, the Mordell-Weil group is isomorphic to $({\bf Z}/3{\bf Z})^4$.
\end{remark}

\subsection{Quasi-elliptic surfaces with a section}
\label{a section}

Let $\varphi : X \longrightarrow {\bf P}^1$ be a quasi-elliptic fibration 
of a K3 surface with $\ell$
singular fibers of type ${\rm IV}^*$ $(\ell = 1, 2 ~\mbox{or}~3)$ and
with $10 - 3\ell$ singular fibers of type ${\rm IV}$.
For $i= 1, \ldots, \ell$, $j = 1, 2, 3$ and $k = 1, 2$, let $C_{jk}^{(i)}$ and $C^{(i)}$
denote the irreducible components of the singular fibers of type ${\rm IV}^*$, so that
$$
   \sum_{j=1}^{3}C^{(i)}_{j1} +  2\sum_{j=1}^{3}C^{(i)}_{j2} + 3 C^{(i)}\quad (i = 1, \hdots, \ell)
$$
are the singular fibers of type ${\rm IV}^*$ with $C^{(i)}_{j1}\cdot C^{(i)}_{j2}= C^{(i)}_{j2}\cdot C^{(i)} =1$
and all other intersection numbers zero.
For $i= 1, 2, \ldots, 10 -3 \ell$ and $j = 1, 2, 3$, let $E_j^{(i)}$ 
denote the irreducible components of the singular fibers of type ${\rm IV}$,
so that $E_1^{(i)}$, $E_2^{(i)}$, $E_3^{(i)}$ together form the $i$-th singular fiber  as before. 

We denote by $F$ a general fiber and by $\xi$ the cusp locus of the quasi-elliptic fibration, as before.
In this subsection, assume that the quasi-elliptic surface 
has a section $s$ and $s$ intersects $E^{(i)}_{1}$ $(i = 1, \ldots , 10 - \ell)$
and $C^{(i)}_{11}$ $(i = 1,  \ldots, \ell)$. 

\begin{theorem}
In ${\rm NS}(X)\otimes {\bf Q}$, we have 
$$
\xi = -\sum_{i = 1}^{10 -3\ell}(E^{(i)}_{2}+ E^{(i)}_{3}) + 
\sum_{i=1}^{\ell}(2C^{(i)}_{11} + C^{(i)}_{12}) + 3s + (6 - 2\ell)F.
$$
\end{theorem}

\begin{proof}
The classes
$E_j^{(i)}$ $(j = 2, 3; i = 1, 2, \ldots,10 -3 \ell)$, 
$C_{jk}^{(i)}$ $(j = 1, 2, 3; k = 1, 2; i=1, \ldots, \ell)$, together with $s$ and $F$, 
form a basis of ${\rm NS}(X)\otimes {\bf Q}$.
Therefore, $\xi$ can be written as a linear combination of these elements as follows:
$$
\xi = \sum_{i = 1}^{10 -3\ell}(\alpha^{(i)}_2E^{(i)}_{2}+ \alpha^{(i)}_3E^{(i)}_{3}) + 
\sum_{i=1}^{\ell}\sum_{j=1}^{3}(\beta^{(i)}_{j1}C^{(i)}_{j1} + \beta^{(i)}_{j2}C^{(i)}_{j2}) + cs + dF
$$
where $\alpha^{(i)}_r, \beta^{(i)}_{jk}, c, d \in {\bf Q}$.
Since $3 = \xi \cdot F = c(s\cdot F) = c$, we have $c= 3$. 
By $1 = E_{2}^{(i)}\cdot \xi = -2\alpha_2^{(i)} + \alpha_3^{(i)}$ 
and $1 = E_{3}^{(i)}\cdot \xi = \alpha_2^{(i)} - 2\alpha_3^{(i)}$,
we have $\alpha_2^{(i)}= \alpha_3^{(i)} = -1$.
By $0 = C^{(i)}_{11}\cdot \xi = -2\beta_{11}^{(i)} + \beta_{12}^{(i)} + c$ 
and $0 = C^{(i)}_{12}\cdot \xi = \beta_{11}^{(i)} - 2\beta_{12}^{(i)}$,
we have $\beta_{11}^{(i)}= 2$ and $\beta_{12}^{(i)} = 1$.
By $0 = C^{(i)}_{21}\cdot \xi = -2\beta_{21}^{(i)} + \beta_{22}^{(i)}$ 
and $0 = C^{(i)}_{22}\cdot \xi = \beta_{21}^{(i)} - 2\beta_{22}^{(i)}$,
we have $\beta_{21}^{(i)}= \beta_{22}^{(i)} = 0$.
Similarly, we have $\beta_{31}^{(i)}= \beta_{32}^{(i)} = 0$.
Since $\xi^2 = -2$, we have $d = 6 - 2\ell$.
\end{proof}

\begin{corollary}\label{m}
Under the above notation, assume $\ell =0$, that is, all singular fibers are of type $IV$.
Then any two sections intersect the same component of a singular fiber for exactly 
one or four fibers.
\end{corollary}
\begin{proof}
In this case, we have
\begin{equation}\label{s}
    \xi =  -\sum_{i = 1}^{10}(E^{(i)}_{2}+ E^{(i)}_{3}) + 3s + (6 - 2\ell)F.
\end{equation}
Let $s'$ be another section. Let $s$ and $s'$ intersect the same component of a singular fiber
exactly $m$ times. We may assume that $s'$ intersects 
$E_1^{(i)}$ $(i = 1, 2, \ldots, m)$ and $E_2^{(i)}$ $(i = m+1, \ldots, 10)$.
Then we have
\begin{equation}\label{s'}
\xi = -\sum_{i = 1}^{m}(E^{(i)}_{2}+ E^{(i)}_{3}) -\sum_{i = m+ 1}^{10}(E^{(i)}_{1}+ E^{(i)}_{3})+ 3s' + (6 - 2\ell)F.
\end{equation}
Subtracting (\ref{s'}) from (\ref{s}), we obtain
$$
  0 = \sum_{i = m+ 1}^{10}(E^{(i)}_{1} + 2E^{(i)}_{2}) - 3\sum_{i = m+ 1}^{10}E^{(i)}_{2}
  + 3(s - s').
$$
Therefore, $\sum_{i = m+ 1}^{10}(2E^{(i)}_{1} + E^{(i)}_{2})$ is divisible by 3 in ${\rm NS}(X)$.
Since the two sections do not intersect the same irreducible component in all singular fibers
(cf. Miyanishi-Ito \cite[Chapter 4, 4.1.3]{MI}),
we have $10 - m \neq 0$.
By Corollary \ref{6,9}, this implies that $10 - m = 6$ or $9$, and hence $m=4$ or $1$.
\end{proof}

The following corollary is an immediate consequence of Corollary \ref{m}.

\begin{corollary}\label{mm}
Under the notation of Corollary \ref{m}, the following holds.
If $m=1$, then 
$\sum_{i = 2}^{10}(E^{(i)}_{1} + 2E^{(i)}_{2})$ 
and $\sum_{i = 2}^{10}(2E^{(i)}_{1} + E^{(i)}_{2})$ are divisible by $3$ in ${\rm NS}(X)$.
If $m=4$, then 
$\sum_{i = 5}^{10}(E^{(i)}_{1} + 2E^{(i)}_{2})$ and
$\sum_{i = 5}^{10}(2E^{(i)}_{1} + E^{(i)}_{2})$ are divisible by $3$ in ${\rm NS}(X)$.
\end{corollary}

\begin{remark}
Concrete examples of Corollary \ref{mm} are given in \cite{I1}. However,
the paper \cite{I1} contains two misprints in p.~36, Fig. 13 ${\bf a}$, ${\bf b}$.
In Fig. ${\bf a}$, the intersection of the section $P$ with the rightmost fiber is the same as 
its intersection with the preceding eight fibers. In Fig. ${\bf b}$, on the leftmost fiber, the section $Q$ does not intersect the irreducible component met by the section $O$.
\end{remark}

\begin{remark}
Corollary \ref{mm} can also be derived from the theory of Mordell-Weil lattices \cite{SS},
applied to (torsion) sections of quasi-elliptic fibrations.
\end{remark}

\subsection{Quasi-elliptic surfaces without a section}
\label{no-section}

In this subsection, we consider the case where the quasi-elliptic surface 
$\varphi : X \longrightarrow {\bf P}^1$
has 10 singular fibers of type ${\rm IV}$.
We consider the projective model from Subsection \ref{structure}
and use the notation therein.

\begin{definition}
A nonsingular irreducible curve $C$ is called a 3-section if $C^2 = -2$ and $C\cdot F = 3$.
\end{definition}

Since $X$ is a K3 surface, any 3-section is a nonsingular rational curve.
Of course, for a quasi-elliptic fibration in characteristic $3$, 
the curve of cusps is the standard example of a 3-section.
The reader should be aware that, as opposed to a 3-section,
we make no restrictions on the arithmetic genus of an arbitrary trisection below.

\begin{theorem}\label{3-section}
Under the above assumptions, suppose moreover that there exists a 3-section $C$
such that 
$$
\begin{array}{l}
  C\cdot E_1^{(i)}=1, C\cdot E_2^{(i)}= 0, C\cdot E_3^{(i)}= 2\quad (i = 1, 2, \ldots, 9),\\
  C\cdot E_j^{(10)}= 1\quad (j=1,2,3).
\end{array}
$$
Then, in  ${\rm NS}(X)\otimes {\bf Q}$, we have 
$$
\xi = - \frac{1}{3}\sum_{i = 1}^{9}(E^{(i)}_{1}+ 2E^{(i)}_{2}) + C + 2F.
$$
In particular, the divisor $\sum_{i = 1}^{9}(E^{(i)}_{1}+ 2E^{(i)}_{2})$ is divisible by 3 in ${\rm NS}(X)$.
\end{theorem}

\begin{proof}
The classes $E_j^{(i)}$ $(j = 1, 2; i = 1, 2, \ldots,10)$ together with $C$ and $F$, 
form a basis of ${\rm NS}(X)\otimes {\bf Q}$.
Therefore, $\xi$ can be written in the form:
$$
\xi = \sum_{i = 1}^{10}(\alpha^{(i)}_1E^{(i)}_{1}+ \alpha^{(i)}_2E^{(i)}_{2}) + 
aC + bF,
$$
where $\alpha^{(i)}_r, a, b \in {\bf Q}$.
Since $3 = \xi \cdot F = a(C\cdot F) = 3a$, we have $a= 1$. 
From $1 = E_{1}^{(10)}\cdot \xi = -2\alpha_1^{(10)} + \alpha_2^{(10)}+1$ and
$1 = E_{2}^{(10)}\cdot \xi = \alpha_1^{(10)} - 2\alpha_2^{(10)}+1$,
we obtain $\alpha_1^{(10)}= \alpha_2^{(10)} = 0$.
For $i = 1, 2, \ldots, 9$, we have
$1 = E^{(i)}_{1}\cdot \xi = -2\alpha_{1}^{(i)} + \alpha_{2}^{(i)} + 1$ 
and $1 = E^{(i)}_{2}\cdot \xi = \alpha_{1}^{(i)} - 2\alpha_{2}^{(i)}$ and hence
$\alpha_{1}^{(i)}= - \frac{1}{3}$ and $\alpha_{2}^{(i)} = -\frac{2}{3}$.
Finally, since $\xi^2 = -2$, we have $b = 2$.
\end{proof}

We also have the following converse result:

\begin{theorem}\label{converse}
Under the above notation, assume that the quasi-elliptic K3 surface 
$\varphi : X \longrightarrow {\bf P}^1$ with 10 singular fibers of type ${\rm IV}$ has no section.
If $\sum_{i = 1}^{9}(E_1^{(i)} + 2E_2^{(i)})$ is divisible by 3, then there exists a 3-section $C$
such that $C\cdot E_1^{(i)} = 1$, $C\cdot E_2^{(i)} = 0$, $C\cdot E_j^{(10)} = 1$ $(j =1, 2, 3)$.
Here the linear system $|H|$ realizes $C$ as a smooth rational curve of degree $4$ in $\bbP^3$.
\end{theorem}

\begin{proof}
We set $D = \frac{1}{3}\sum_{i = 1}^{9}(E_1^{i} + 2E_2^{i})$ and $G = \xi + D -2F$.
Then we have 
$$
\begin{array}{l}
D^2 = -6, D\cdot E_1^{(i)}= 0, D\cdot E_2^{(i)}= -1, 
D\cdot E_3^{(i)}= 1 \quad (i = 1,2, \ldots, 9)\\
D \cdot \xi = 9, D\cdot F = 0, D \cdot E_j^{(10)} = 0 \quad (j = 1, 2, 3),
\end{array}
$$
and
$$
\begin{array}{l}
G^2 = -2, G\cdot E_1^{(i)} = 1, G\cdot E_2^{(i)} = 0, 
G\cdot E_3^{(i)} = 2\quad(i = 1,2, \ldots, 9),\\
G\cdot E_j^{(10)} = 1\quad (j = 1, 2, 3).
\end{array}
$$
By the Riemann-Roch theorem, we may assume that $G$ is an effective divisor, and write
$G = C + \sum_{i= 1}^{\ell} F_i$, where $C$ is the horizontal divisor
and the $F_i$ are irreducible components of fibers.
Since $C\cdot F = G \cdot F = 3$ and $\varphi : X \longrightarrow {\bf P}^1$ has
no section, it follows that $C$ is irreducible. 

Suppose $C = \xi$. Then we have $D = 2F + \sum_{i= 1}^{\ell} F_i$.
Since ${\rm H}^0({\mathcal O}_X(3D)) = 0$, we have ${\rm H}^0({\mathcal O}_X(D)) = 0$.
However, we have $\dim {\rm H}^0({\mathcal O}_X(2F + \sum F_i)) \geq 3$, a contradiction.
Therefore, $C \neq \xi$, and hence $C \cdot \xi \geq 0$.

Now consider the intersection of $G$ with $\xi$:
$$
   1 = \xi \cdot G = \xi \cdot C + \sum_{i = 1}^{\ell} \xi \cdot F_i.
$$
Since $\xi \cdot F_i \geq 1$ for $i = 1, \ldots, \ell$ and $C \neq \xi$, 
there are only two possibilities:

Case (1). $G = C$ with $\xi \cdot C = 1$ and $\ell = 0$,

Case (2). $G = C + F_1$ with $\xi \cdot C = 0$ and $\xi \cdot F_1 = 1$ $(\ell = 1)$,

Suppose that Case (2) occurs. Then $F_1$ is one of the curves
$E_j^{(10)}$, $E_1^{(i)}$, $E_2^{(i)}$ or $E_3^{(i)}$,
and $C = G - F_1$. In any case, we have $C^2 \leq -4$.
However, no such irreducible curve exists on a K3 surface.
Therefore Case (1) must occur, and 
we obtain the desired 3-section $C$.
\end{proof}

By Theorems \ref{3-section} and \ref{converse}, we have the following result.

\begin{corollary}
Under the above notation, assume that the quasi-elliptic K3 surface 
$\varphi : X \longrightarrow {\bf P}^1$ has no section. Then
$\sum_{i = 1}^{9}(E_1^{(i)} + 2E_2^{(i)})$ is divisible by 3 if and only if
there exists a 3-section $C$
such that $C\cdot E_1^{(i)} =1$, $C\cdot E_2^{(i)} =0$ for all $i=1,\hdots,9$
and  $C\cdot E_j^{(10)} =1$ $(j =1, 2, 3)$.
\end{corollary}

For the case $n=6$, the analogous arguments to Theorem \ref{3-section}
give the following result:

\begin{theorem}
\label{thm:C6}
Suppose the quasi-elliptic surface 
$X$
with 10  fibers
of type ${\rm IV}$ admits
 a trisection $C$ of arithmetic genus $p_a(C)=1$
such that 
$$
\begin{array}{l}
  C\cdot E_1^{(i)}=1, C\cdot E_2^{(i)}= 0, C\cdot E_3^{(i)}= 2\quad (i = 1, 2, \ldots, 6),\\
  C\cdot E_j^{(i)}= 1\quad (i=7,\hdots,10; j=1,2,3).
\end{array}
$$
Then, in  ${\rm NS}(X)\otimes {\bf Q}$, we have 
$$
\xi = - \frac{1}{3}\sum_{i = 1}^{6}(E^{(i)}_{1}+ 2E^{(i)}_{2}) + C + F.
$$
In particular, the divisor $\sum_{i = 1}^{6}(E^{(i)}_{1}+ 2E^{(i)}_{2})$ is divisible by 3 in ${\rm NS}(X)$.
\end{theorem}

There also is the following converse:

\begin{theorem}\label{converse-6}
In the above setting, assume that the quasi-elliptic fibration  
$\varphi : X \longrightarrow {\bf P}^1$ with 10 singular fibers of type ${\rm IV}$ has no section.
If $\sum_{i = 1}^{6}(E_1^{(i)} + 2E_2^{(i)})$ is divisible by 3.
Then there exist a trisection $C$
as in Theorem \ref{thm:C6} which is realized as a $(2,2)$ intersection in $\bbP^3$ by the linear system $|H|$.
\end{theorem}

\begin{proof}
Consider the divisor
\[
D = \frac{1}{3}\sum_{i = 1}^{6}(E^{(i)}_{1}+ 2E^{(i)}_{2}) + \xi - F.
\]
By construction, this is isotropic,
hence effective by Riemann--Roch as $D.F=3>0$.
Spelling it out into irreducible components
\[
D = C + \Theta_1+\hdots + \Theta_r,
\]
we may assume that $C$ is horizontal, hence a trisection of arithmetic genus $p_a(C)=1$ 
(where irreducibility follows from the assumption that there are no sections)
while the $\Theta_i$ are fibre components, in particular smooth rational curves.
Note that $D.H=4$.
The arithmetic genus of $C$ thus predicts, thanks to Theorem \ref{thm:H}, that $r\leq 1$.
More to the point, if $r=1$, then $C$ is embedded in $\bbP^3$ a planar cubic via $|H|$.
As such, it can only meet one line with multiplicity greater than one,
but not $E^{(i)}_3$ for $i=1\hdots,6$
(less one curve if it equals $\Theta_1$).
This gives the required contradiction and completes the proof of the theorem.
\end{proof}

To wrap up the discussion, we record the following direct consequences:

\begin{corollary}
Under the above notation, assume that the quasi-elliptic K3 surface 
$\varphi : X \longrightarrow {\bf P}^1$ has no section. Then
$\sum_{i = 1}^{6}(E_1^{(i)} + 2E_2^{(i)})$ is divisible by 3 if and only if
there exists a trisection $C$ of arithmetic genus $p_a(C)=1$
such that $C\cdot E_1^{(i)} =1$, $C\cdot E_2^{(i)} =0$ for all $i=1,\hdots,6$
and  $C\cdot E_j^{(i)} =1$ $(i=7,\hdots,10; j =1, 2, 3)$.
\end{corollary}

\begin{corollary}
\label{cor:converse}
Assume that the quasi-elliptic fibration  
$\varphi : X \longrightarrow {\bf P}^1$ with 10 singular fibers of type ${\rm IV}$ has no section,
but a trisection meeting exactly two components of some given fibre.
Then, up to renumbering fibres and components, 
there is another trisection as in Theorem \ref{converse} or \ref{converse-6}.
\end{corollary}

\begin{proof}
Assume that the trisection $C$ satisfies $ C\cdot E_1^{(1)}=1, C\cdot E_2^{(2)}= 0, C\cdot E_3^{(3)}= 2$.
Then $C$ is expressed in the standard basis of $\NS(X)$ as
\begin{eqnarray}
\label{eq:CCC}
C = \xi + \frac 13 (E^{(1)}_{1}+ 2E^{(1)}_{2})
+ \sum_{i = 2}^{10}(a_iE^{(i)}_{1}+ b_iE^{(i)}_{2})
+ NF.
\end{eqnarray}
Here the integer intersection numbers with any curve on $X$ not only predict that the coefficients $a_i, b_i, N$ lie in $\frac 13\bbZ$,
but also give the extra conditions
\[
a_i+b_i \in \bbZ \;\; \forall i=2,\hdots,10 \;\;\; \text{ and } \;\;\; N -\frac l3 \in \bbZ
\]
where $l-1=\#\{i\in\{2,\hdots,10\}; a_i\not\in\bbZ\}$ and the last relation guarantees that $C^2\in\bbZ$.

If $N\in\bbZ$, then \eqref{eq:CCC} directly gives a $3$-divisible $A_2^l$ configuration,
whence $l=6$ or $9$ by Corollary \ref{6,9} and we conclude using Theorem \ref{converse} or \ref{converse-6}.

If $N\not\in\bbZ$ and $l\equiv 1$ mod  $3$, then absorbing $\frac 13 F$ 
into the leftmost non-integer term transforms \eqref{eq:CCC} to
\begin{eqnarray}
\label{eq:CCC'}
\;\;\; C = \xi + \frac 13 (2E^{(1)}_{1}+ E^{(1)}_{3})+E^{(1)}_2
+ \sum_{i = 2}^{10}(a_iE^{(i)}_{1}+ b_iE^{(i)}_{2})
+ \left(N-\frac 13\right) F.
\end{eqnarray}
Here $N-\frac 13\in\bbZ$,
so, again, we obtain a $3$-divisible $A_2^l$ configuration and conclude as before.
The analogous case $l\equiv 2$ mod $3$ completes the proof.
\end{proof}


\end{document}